\documentclass{article}
\usepackage[utf8]{inputenc}
\usepackage{graphicx}
\usepackage{amsfonts}
\usepackage{amsmath}
\usepackage{amssymb}
\usepackage{fancyhdr}
\usepackage{titlesec}
\usepackage{indentfirst}
\usepackage{booktabs}
\usepackage{verbatim}
\usepackage{color}
\usepackage{amsthm}
\usepackage{subfigure}
\usepackage{mathabx}

\usepackage[page,header]{appendix}
\usepackage{titletoc}

\newcommand{\rd}{\,\mathrm{d}}
\numberwithin{equation}{section}
\newtheorem{theorem}{Theorem}[section]
\newtheorem{lemma}[theorem]{Lemma}

\usepackage{mathtools}

\def\bx{{\bf x}}
\def\by{{\bf y}}

\def\cM{\mathcal{M}}

\def\supp{\textnormal{supp\,}}

\begin{document}

\title{Break of radial symmetry for a class of attractive-repulsive interaction energy minimizers}

\author{ Ruiwen Shu\footnote{Department of Mathematics, University of Georgia, Athens, GA 30602 (ruiwen.shu@uga.edu).}}

\maketitle

\abstract{Break of radial symmetry for interaction energy minimizers is a phenomenon where a radial interaction potential whose associated energy minimizers are never radially symmetric. Numerically, it has been frequently observed  for various types of interaction potentials, however,  rigorous justification of this phenomenon was only done in very limited cases. We propose a new approach to prove  the break of radial symmetry, by using a lower bound for the energy in the class of radial probability measures, combining with the construction of a probability measure whose energy is lower than this lower bound. In particular, we prove that for a class of interaction potentials that are repulsive at short distance and attractive at long distance, every energy minimizer is necessarily a H\"older continuous function which is not radially symmetric.}

\section{Introduction}

In this paper we study the minimizers of the interaction energy 
\begin{equation}\label{E}
	E[\rho] = \frac{1}{2}\int_{\mathbb{R}^d}\int_{\mathbb{R}^d}W(\bx-\by)\rd{\rho(\by)}\rd{\rho(\bx)}\,,
\end{equation}
where $\rho\in \cM(\mathbb{R}^d)$ is a probability measure. $W:\mathbb{R}^d\rightarrow \mathbb{R}\cup\{\infty\}$ is an interaction potential, which satisfies the basic assumption
\begin{equation}\begin{split}
		\text{{\bf (W0)}: $W$} & \text{ is locally integrable, bounded from below} \\ & \text{and lower-semicontinuous, with $W(\bx)=W(-\bx)$.}
\end{split}\end{equation}
As a consequence, $E[\rho]$ is always well-defined for any $\rho\in \cM(\mathbb{R}^d)$, taking values in $\mathbb{R}\cup \{\infty\}$.

The existence of minimizers of $E$ is well-understood: (see also similar results in \cite{CCP15,CFT15})
\begin{theorem}[\cite{SST15}]\label{thm_exist2}
	Assume {\bf (W0)}. Assume $\lim_{|\bx|\rightarrow\infty}W(\bx)=:W_\infty$ is a real number or $\infty$. If there exists $\rho\in \cM(\mathbb{R}^d)$ such that $E[\rho]< \frac{1}{2}W_\infty$, i.e., 
	\begin{equation}\label{thm_exist2_1}
		\inf_{\rho\in\cM(\mathbb{R}^d)}E[\rho] < \frac{1}{2}W_\infty\,,
	\end{equation}
	then there exists a minimizer of $E$.
\end{theorem}
The uniqueness of minimizers is also studied by many recent works \cite{Lop19,CS21,ST,DLM1,DLM2,Fra}, mainly using convexity theory. We say $W$ has the linear interpolation convexity (LIC), if for any distinct compactly-supported $\rho_0,\rho_1\in\cM(\mathbb{R}^d)$ such that $\int_{\mathbb{R}^d}\bx\rho_0(\bx)\rd{\bx}=\int_{\mathbb{R}^d}\bx\rho_1(\bx)\rd{\bx}$, the function $t\mapsto E [(1-t)\rho_0+t\rho_1]$ is strictly convex on $t\in [0,1]$. For LIC potentials, the energy minimizer is necessarily unique (up to translation), and the LIC property is closely related to the positivity of $\hat{W}$, the Fourier transform of $W$. Based on convexity theory, explicit formulas for the minimizer are also obtained for some power-law potentials, which include minimizers as a locally integrable function \cite{CH17,CS21,Fra}, spherical shells \cite{BCLR13_2,DLM1,FM} and discrete measures \cite{DLM2}.

In this paper we focus on the \emph{break of radial symmetry} of energy minimizers, i.e., the following question:
\begin{quote}
	Let $d\ge 2$. If $W$ is radial, when can one guarantee that no minimizer of $E$ is  radially symmetric?
\end{quote}
Numerical experiments have shown that the break of radial symmetry is not at all a rare phenomenon \cite{BKSUB,KSUB,CS21}, and one can observe the formation of clusters or fractal structures for minimizers. On the contrast, only a few previous works give rigorous proof of the break of radial symmetry, which include
\begin{itemize}
	\item \cite{CFP17} showed that if $W(\bx)\sim -|\bx|^{b}$ with $b>2$ for small $|\bx|$, then any minimizer is a finite linear combination of Dirac masses. This result is proved by a delicate local variational argument. In particular, as long as one can guarantee that a single Dirac mass is not a minimizer, then no minimizer is radially symmetric. 
	\item \cite{DLM2} showed that for the power-law potential $W(\bx)=\frac{|\bx|^a}{a}-\frac{|\bx|^b}{b}$, if $a>b$, $a\ge 4$, $b\ge 2$, $(a,b)\ne (4,2)$, then the unique energy minimizer (up to translation and rotation) is the Dirac masses with equal mass located at the vertices of a unit simplex. This is proved by a careful study of the borderline case $(a,b)= (4,2)$ (which has a non-strict LIC and infinitely many minimizers), combined with a comparison argument for the energies with different $a,b$.
	\item \cite{BCHC} showed that for the power-law potential $W(\bx)=\frac{|\bx|^a}{a}-\frac{|\bx|^b}{b}$, for every fixed $b<0$, if $d$ and $a$ are sufficiently large, then no minimizer is radially symmetric. This was obtained by considering the asymptotic limit $a\rightarrow \infty$ in the sense of $\Gamma$-convergence, and justifying the asymmetry in the limiting case via comparison arguments. The largeness requirement of $a$ is not known in a  quantitative sense since the asymptotic limit $a\rightarrow \infty$ is not.
\end{itemize}
To the best of the author's knowledge, there is no previously-known explicit example of $W$ such that the break of radial symmetry can be justified while the energy minimizers are not discrete. 

The major difficulty to rigorously study the break of radial symmetry is that the convexity techniques cannot be directly applied. In fact, for a radial potential $W(\bx) = w(|\bx|)$, if the minimizer is unique (up to translation), then it is automatically radial, because otherwise a rotation of a minimizer will give a different minimizer. There have been some efforts in using the contrary of convexity to study the break of radial symmetry. \cite{BKSUB,KSUB} gave a bifurcation analysis for the break of a ring-shaped minimizer into 3 pieces, for $W(\bx)=\frac{|\bx|^a}{a}-\frac{|\bx|^b}{b}$. \cite{CS21} showed that a large region of negative values of $\hat{W}$ will prohibit the minimizers to have large clusters. However, these approaches do not lead to a rigorous result that break of radial symmetry must happen.

In this paper we propose a simple strategy to rigorously prove the break of radial symmetry.
\begin{lemma}\label{lem_radmin}
	Under the same assumptions as Theorem \ref{thm_exist2}, further assume that $d\ge 2$ and $W$ is radial. If there exists $\rho_*\in \cM(\mathbb{R}^d)$ such that
	\begin{equation}
		E[\rho_*] < \inf_{\rho\in \cM(\mathbb{R}^d),\,\textnormal{radial}} E[\rho]\,,
	\end{equation}
	then no minimizer of $E$ is radially symmetric.
\end{lemma}
Although the proof of this lemma is straightforward, it allows us to justify the break of radial symmetry via the following two steps:
\begin{itemize}
	\item Give a lower bound of $\inf_{\rho\in \cM(\mathbb{R}^d),\,\textnormal{radial}} E[\rho]$.
	\item Construct $\rho_*\in \cM(\mathbb{R}^d)$ whose energy is less than the above lower bound.
\end{itemize}

Our main result (Theorem \ref{thm_nonradex}) states that there exists a radial potential $W$, smooth on $\mathbb{R}^d\backslash \{0\}$, repulsive at short distance and attractive at long distance, such that every minimizer of $E$ is a H\"older continuous function which is not radially symmetric.

To prove this result, we start in Sections 2 and 3 by analyzing $\inf_{\rho\in \cM(\mathbb{R}^d),\,\textnormal{radial}} E[\rho]$ carefully for a simple potential, giving a lower bound, and show that $\rho_*$ being the Dirac masses at the vertices of a unit simplex has energy less than this lower bound. Then, in Section 4, we use a comparison argument to show that for any potential which is slightly larger than this special one, break of radial symmetry also happens. Among this class of potentials, we construct one that satisfies the prescribed properties (where the H\"older continuity of minimizer is obtained by using \cite{CDM}).

\section{Energy for radial measures}
Assume $d\ge 2$ and consider a radial potential $W(\bx)=w(|\bx|)$ satisfying {\bf (W0)}. Take a radial $\rho\in \cM(\mathbb{R}^d)$. Then there exists $\tilde{\rho}\in\cM([0,\infty))$ such that
\begin{equation}
	\rho = \int_{[0,\infty)}\delta_{(t)}\rd{\tilde{\rho}(t)}\,,
\end{equation}
where $\delta_{(t)},\,t>0$ denotes the uniform distribution on $\partial B(0;t)$ with total mass 1, and $\delta_{(0)}$ denotes the Dirac mass at 0. Then for $r,s\ge 0$, we define
\begin{equation}
	\tilde{W}(r,s) := \int_{\mathbb{R}^d}\int_{\mathbb{R}^d}W(\bx-\by)\rd{\delta_{(s)}(\by)}\rd{\delta_{(r)}(\bx)}\,.
\end{equation}
It is clear that $\tilde{W}(r,s)=\tilde{W}(s,r)$, $\tilde{W}(r,0)=w(r)$. Then the energy for a radial $\rho$ is
\begin{equation}
	E[\rho] = \frac{1}{2}\int_{[0,\infty)}\int_{[0,\infty)}\tilde{W}(r,s)\rd{\tilde{\rho}(s)}\rd{\tilde{\rho}(r)}\,.
\end{equation}

For $r\ge s> 0$, we calculate
\begin{equation}\begin{split}
		\tilde{W}(r,s)  
		= & \frac{1}{|S^{d-1}|}\int_{S^{d-1}} w(|r\vec{e}_1-s\hat{\by}|)\rd{S(\hat{\by})} \\
		= & \frac{1}{|S^{d-1}|}\int_0^\pi \int_{S^{d-2}} w(|r\vec{e}_1-s(\cos\phi,\by'\sin\phi)|)\rd{S(\by')} \sin^{d-2}\phi\rd{\phi} \\
		= & \frac{1}{|S^{d-1}|}\int_0^\pi \int_{S^{d-2}} w\Big(\sqrt{r^2-2rs\cos\phi+s^2}\Big)\rd{S(\by')} \sin^{d-2}\phi\rd{\phi} \\
		= & \frac{|S^{d-2}|}{|S^{d-1}|}\int_0^\pi w\Big(\sqrt{r^2-2rs\cos\phi+s^2}\Big)\sin^{d-2}\phi\rd{\phi}\,,
\end{split}\end{equation}
where $\vec{e}_1=(1,0,\dots,0)\in\mathbb{R}^d$, $\rd{S(\cdot)}$ denotes the surface measure on a sphere, and we used the spherical coordinates $\hat{\by} = (\cos\phi,\by'\sin\phi)$. Then, using the change of variables
\begin{equation}
	t = \sqrt{r^2-2rs\cos\phi+s^2},\, \phi = \cos^{-1}\Big(\frac{r^2+s^2-t^2}{2rs}\Big),\, \frac{\rd{\phi}}{\rd{t}} = \Big(1-\big(\frac{r^2+s^2-t^2}{2rs}\big)^2\Big)^{-1/2}\cdot \frac{t}{rs}\,,
\end{equation}
we get
\begin{equation}\begin{split}
		\tilde{W}(r,s) = & \int_0^\infty K_{r,s}(t) w(t)\rd{t}\,,
\end{split}\end{equation}
where the kernel function $K_{r,s}(t)$ is given by
\begin{equation}
	K_{r,s}(t) = \frac{|S^{d-2}|}{|S^{d-1}|}\Big(1-\big(\frac{r^2+s^2-t^2}{2rs}\big)^2\Big)^{(d-3)/2} \cdot \frac{t}{rs}\cdot \chi_{(r-s,r+s)}(t)\,.
\end{equation}
It is clear that $\int_0^\infty K_{r,s}(t) \rd{t}=1$. 

\section{A prototype potential}

Let $0<\epsilon<1$. Define
\begin{equation}
	W_\epsilon(\bx) = w_\epsilon(|\bx|) = \left\{\begin{split}
		-1,& \quad \big| |\bx|-1\big| \le \epsilon \\
		0,& \quad \textnormal{otherwise} \\
	\end{split}\right.\,,
\end{equation}
which satisfies the assumptions of Theorem \ref{thm_exist2}, and the corresponding energy is denoted as $E_\epsilon$. We first give a lower bound of $E_\epsilon$ for radial probability measures.
\begin{lemma}\label{lem_Weps}
	Assume $d\ge 2$. For $W_\epsilon$ given above,
	\begin{equation}
		\inf_{\rho\in \cM(\mathbb{R}^d),\,\textnormal{radial}} E_\epsilon[\rho] \ge -\frac{1}{4} \Big(1- \frac{1}{2}\sup_{r\ge s \ge \frac{1-\epsilon}{2}} \int_{1-\epsilon}^{1+\epsilon}K_{r,s}(t)\rd{t}\Big)^{-1}\,.
	\end{equation}
\end{lemma}

\begin{proof}
	For a radial $\rho$, we have
	\begin{equation}\begin{split}
			E_\epsilon[\rho] = & \frac{1}{2}\int_{[0,\infty)}\int_{[0,\infty)}\tilde{W}_\epsilon(r,s)\rd{\tilde{\rho}(s)}\rd{\tilde{\rho}(r)} \\
			= & \int_{[0,\frac{1-\epsilon}{2})}\int_{[\frac{1-\epsilon}{2},\infty)}\tilde{W}_\epsilon(r,s)\rd{\tilde{\rho}(s)}\rd{\tilde{\rho}(r)} + \frac{1}{2}\int_{[0,\frac{1-\epsilon}{2})}\int_{[0,\frac{1-\epsilon}{2})}\tilde{W}_\epsilon(r,s)\rd{\tilde{\rho}(s)}\rd{\tilde{\rho}(r)} \\
			& + \frac{1}{2}\int_{[\frac{1-\epsilon}{2},\infty)}\int_{[\frac{1-\epsilon}{2},\infty)}\tilde{W}_\epsilon(r,s)\rd{\tilde{\rho}(s)}\rd{\tilde{\rho}(r)} \,.\\
	\end{split}\end{equation}
	First notice that $\tilde{W}_\epsilon(r,s)\ge -1$ for any $r,s$ since $w_\epsilon\ge -1$. Therefore
	\begin{equation}\begin{split}
			\int_{[0,\frac{1-\epsilon}{2})}\int_{[\frac{1-\epsilon}{2},\infty)}\tilde{W}_\epsilon(r,s)\rd{\tilde{\rho}(s)}\rd{\tilde{\rho}(r)} \ge -m(1-m)\,,
	\end{split}\end{equation}
	where 
	\begin{equation}
		m = \tilde{\rho}\Big([\frac{1-\epsilon}{2},\infty)\Big)\,.
	\end{equation}
	Then notice that if $r,s\in [0,\frac{1-\epsilon}{2})$ then $r+s<1-\epsilon$, which implies $\tilde{W}_\epsilon(r,s)=0$. Finally, since $\tilde{W}_\epsilon(r,s)=-\int_{1-\epsilon}^{1+\epsilon}K_{r,s}(t)\rd{t}$ for any $r\ge s>0$, it is clear that
	\begin{equation}\begin{split}
			\frac{1}{2}\int_{[\frac{1-\epsilon}{2},\infty)}\int_{[\frac{1-\epsilon}{2},\infty)}\tilde{W}_\epsilon(r,s)\rd{\tilde{\rho}(s)}\rd{\tilde{\rho}(r)} \ge -c_0 m^2\,,
	\end{split}\end{equation}
	where $c_0=\frac{1}{2}\sup_{r\ge s \ge \frac{1-\epsilon}{2}} \int_{1-\epsilon}^{1+\epsilon}K_{r,s}(t)\rd{t} \in [0,\frac{1}{2}]$. Combining these, we get
	\begin{equation}\begin{split}
			E_\epsilon[\rho] \ge -m(1-m) -c_0 m^2 = (1-c_0)m^2 - m \ge -\frac{1}{4(1-c_0)}\,,
	\end{split}\end{equation}
	by minimizing in $m\in [0,1]$, which is the conclusion.
\end{proof}

Then we give an estimate on the supremum of a partial integral of $K_{r,s}$.
\begin{lemma}\label{lem_krs}
	Assume $d\ge 2$, $0<\epsilon<1$. Then $K_{r,s}$ satisfies the estimate
	\begin{equation}
		\sup_{r\ge s \ge \frac{1-\epsilon}{2}} \int_{1-\epsilon}^{1+\epsilon}K_{r,s}(t)\rd{t} \le \left\{\begin{split}
			\frac{22\sqrt{2}}{5\pi}\sqrt{\epsilon},& \quad \textnormal{for }d=2,\,\epsilon\le \frac{1}{11} \\
			\frac{|S^{d-2}|}{|S^{d-1}|}\cdot \frac{8}{(1-\epsilon)^2}\epsilon,& \quad \textnormal{for } d\ge 3 \\
		\end{split}\right.\,.
	\end{equation}
\end{lemma}
\begin{proof}
	\begin{equation}\label{krsdu}\begin{split}
			\int_{1-\epsilon}^{1+\epsilon}K_{r,s}(t)\rd{t} = & \frac{|S^{d-2}|}{|S^{d-1}|}\cdot \frac{1}{ rs}\int_{1-\epsilon}^{1+\epsilon}\Big(1-\big(\frac{r^2+s^2-t^2}{2rs}\big)^2\Big)_+^{(d-3)/2}t\rd{t}  \\
			= & \frac{|S^{d-2}|}{|S^{d-1}|}\cdot \frac{1}{ 2rs}\int_{(1-\epsilon)^2}^{(1+\epsilon)^2}\Big(1-\big(\frac{r^2+s^2-u}{2rs}\big)^2\Big)_+^{(d-3)/2}\rd{u}\,.  \\
	\end{split}\end{equation}
	If $d\ge 3$, then the last integrand is at most 1. Then for $r\ge s \ge \frac{1-\epsilon}{2}$,
	\begin{equation}\begin{split}
			\int_{1-\epsilon}^{1+\epsilon}K_{r,s}(t)\rd{t} \le \frac{|S^{d-2}|}{|S^{d-1}|}\cdot \frac{1}{ 2(\frac{1-\epsilon}{2})^2}\cdot 4\epsilon = \frac{|S^{d-2}|}{|S^{d-1}|}\cdot \frac{8}{(1-\epsilon)^2}\epsilon\,. \\
	\end{split}\end{equation}

	If $d=2$, then notice that the last integrand in \eqref{krsdu} is supported on $[(r-s)^2,(r+s)^2]$, increasing on $u\in [r^2+s^2,(r+s)^2)$ and symmetric about $u=r^2+s^2$. Also, for $\epsilon\le \frac{1}{11}$ and $r\ge s \ge \frac{1-\epsilon}{2}$, we have
	\begin{equation}
		rs \ge (\frac{1-\epsilon}{2})^2 \ge \frac{25}{121} > 2\epsilon\,,
	\end{equation}
	i.e., the length of the interval $((1-\epsilon)^2,(1+\epsilon)^2)$ is $4\epsilon$ which is less than $2rs$, the length of $[r^2+s^2,(r+s)^2)$. Therefore,  this integral can be estimated by 
	\begin{equation}\begin{split}
			& \int_{(1-\epsilon)^2}^{(1+\epsilon)^2}\Big(1-\big(\frac{r^2+s^2-u}{2rs}\big)^2\Big)_+^{-1/2}\rd{u} \le  \int_{(r-s)^2}^{(r-s)^2+4\epsilon}\Big(1-\big(\frac{r^2+s^2-u}{2rs}\big)^2\Big)^{-1/2}\rd{u} \\
			= & \int_0^{4\epsilon} \Big(1-\big(\frac{r^2+s^2-u-(r-s)^2}{2rs}\big)^2\Big)^{-1/2}\rd{u} = \int_0^{4\epsilon} \big(\frac{u(4rs-u)}{4r^2s^2}\big)^{-1/2}\rd{u} \\
			\le & \int_0^{4\epsilon} \big(\frac{u\cdot 2rs}{4r^2s^2}\big)^{-1/2}\rd{u} = \sqrt{2rs}\cdot 2\sqrt{4\epsilon}\,.
	\end{split}\end{equation}
	Therefore we conclude  
	\begin{equation}\begin{split}
			\int_{1-\epsilon}^{1+\epsilon}K_{r,s}(t)\rd{t} \le & \frac{|S^{0}|}{|S^{1}|}\cdot \frac{1}{ 2rs} \cdot \sqrt{2rs}\cdot 2\sqrt{4\epsilon} = \frac{2\sqrt{2}}{\pi}\frac{1}{\sqrt{rs}} \sqrt{\epsilon} \le \frac{2\sqrt{2}}{\pi}\frac{2}{1-\epsilon} \sqrt{\epsilon} \le \frac{22\sqrt{2}}{5\pi}\sqrt{\epsilon}\,. \\
	\end{split}\end{equation}

\end{proof}

\begin{theorem}\label{thm_nonradeps}
	Assume $d\ge 2$. There exists $\epsilon_0>0$, depending on $d$, such that for any $\epsilon<\epsilon_0$, no minimizer of $E_\epsilon$ is radially symmetric.
\end{theorem}

An explicit choice of $\epsilon_0$ is given in the proof.

\begin{proof}
	
	We construct $\rho_{*,0}\in\cM(\mathbb{R}^d)$, which is a sum of $d+1$ Dirac masses, each having mass $\frac{1}{d+1}$, located at the vertices of a simplex of side length 1. Then it is clear that
	\begin{equation}
		E_\epsilon[\rho_{*,0}] = \frac{1}{2}\sum_{i,j\in\{1,\dots,d+1\},\,i\ne j} -\frac{1}{(d+1)^2} = -\frac{d}{2(d+1)}\,.
	\end{equation}
	Therefore, if 
	\begin{equation}\label{krssmall}
		\sup_{r\ge s \ge \frac{1-\epsilon}{2}} \int_{1-\epsilon}^{1+\epsilon}K_{r,s}(t)\rd{t} < \frac{d-1}{d}\,,
	\end{equation}
	then Lemma \ref{lem_Weps} implies $\inf_{\rho\in \cM(\mathbb{R}^d),\,\textnormal{radial}} E_\epsilon[\rho] > E_\epsilon[\rho_*]$. Then by Lemma \ref{lem_radmin}, we have that no minimizer of $E_\epsilon$ is radially symmetric.
	
	The condition \eqref{krssmall} is clearly satisfied for sufficiently small $\epsilon$, due to Lemma \ref{lem_krs}. In fact, for $d=2$, one can take $\epsilon< \epsilon_0 =( \frac{5\pi}{44\sqrt{2}})^2$; for $d\ge 3$ one can take $\epsilon< \epsilon_0 = \frac{(d-1)|S^{d-1}|}{32d|S^{d-2}|}$ (where one notices that $\epsilon_0<\frac{1}{2}$ due to the fact that $\frac{|S^{d-1}|}{|S^{d-2}|} \le 2$ for $d\ge 3$). This finishes the proof. 
	
\end{proof}

\section{Break of radial symmetry for general potentials}

Let $\bx_{*,0},\dots,\bx_{*,d}$ denote the vertices of a simplex of side length 1. For $\eta>0$, denote
\begin{equation}
	\rho_{*,\eta} = \frac{1}{(d+1)|B(0;\eta)|}\sum_{i=0}^d \chi_{B(\bx_{*,i};\eta)} \in\cM(\mathbb{R}^d)\,,
\end{equation}
as a mollified version of $\rho_{*,0}$ in the proof of Theorem \ref{thm_nonradeps}. Notice that for any $\epsilon<\frac{1}{2}$,
\begin{equation}\label{Eepsrhoeta}
	E_\epsilon[\rho_{*,\eta}] = E_\epsilon[\rho_{*,0}] = -\frac{d}{2(d+1)},\quad \forall \eta\le \frac{\epsilon}{2}\,.
\end{equation}
This is because if $\bx,\by\in B(\bx_{*,i};\eta)$ then $|\bx-\by| < \epsilon < 1-\epsilon$ and $W_\epsilon(\bx-\by)=0$; if $\bx\in B(\bx_{*,i};\eta),\,\by\in B(\bx_{*,j};\eta)$ with $i\ne j$ then $1-\epsilon < |\bx-\by| < 1+\epsilon$ and $W_\epsilon(\bx-\by)=-1$. Fix $\epsilon<\epsilon_0$. For a general potential $W$ satisfying $W(\bx)\ge W_\epsilon(\bx),\,\forall \bx$, we may use $\rho_{*,\eta}$ as a candidate of $\rho_*$ to break the radial symmetry of minimizers.

\begin{theorem}\label{thm_nonradgen}
	Assume $d\ge 2$. Fix $\epsilon<\epsilon_0$, where $\epsilon_0$ is as in Theorem \ref{thm_nonradeps}. Suppose $W=W_\epsilon+W_1$ with $W_1\ge 0$ being radial. Suppose $W$ satisfies the assumptions of Theorem \ref{thm_exist2}, and 
	\begin{equation}\label{thm_nonradgen_1}
		\sup_{|\bx|<1+\eta} \frac{1}{|B(0;\eta)|}\int_{B(\bx;\eta)} W_1(\by)\rd{\by} < \frac{d}{d+1} - \frac{1}{2} \Big(1- \frac{1}{2}\sup_{r\ge s \ge \frac{1-\epsilon}{2}} \int_{1-\epsilon}^{1+\epsilon}K_{r,s}(t)\rd{t}\Big)^{-1}\,,
	\end{equation}
	for some $\eta\in(0,\epsilon/2]$, then no minimizer of $E$ is radially symmetric.
\end{theorem}
Here the RHS of \eqref{thm_nonradgen_1} is positive due to \eqref{krssmall} in the proof of Theorem \ref{thm_nonradeps}.

\begin{proof}
	Denote $E$ as the energy with potential $W$. Since $W_1\ge 0$, we have
	\begin{equation}
		\inf_{\rho\in \cM(\mathbb{R}^d),\,\textnormal{radial}} E[\rho] \ge \inf_{\rho\in \cM(\mathbb{R}^d),\,\textnormal{radial}} E_\epsilon[\rho] \ge -\frac{1}{4} \Big(1- \frac{1}{2}\sup_{r\ge s \ge \frac{1-\epsilon}{2}} \int_{1-\epsilon}^{1+\epsilon}K_{r,s}(t)\rd{t}\Big)^{-1}\,,
	\end{equation}
	by Lemma \ref{lem_Weps}. Denote the LHS of \eqref{thm_nonradgen_1} as $W_{1M}$. Then, by \eqref{Eepsrhoeta}, we have 
	\begin{equation}\begin{split}
			E[\rho_{*,\eta}] = & -\frac{d}{2(d+1)} + \frac{1}{2(d+1)^2}\sum_{i,j=0}^d \int_{\mathbb{R}^d} \Big(W_1*\frac{\chi_{B(\bx_{*,i};\eta)}}{|B(0;\eta)|} \Big)\frac{\chi_{B(\bx_{*,j};\eta)}}{|B(0;\eta)|}  \rd{\bx} \\
			\le & -\frac{d}{2(d+1)} + \frac{W_{1M}}{2(d+1)^2}\sum_{i,j=0}^d \int_{\mathbb{R}^d} \frac{\chi_{B(\bx_{*,j};\eta)}}{|B(0;\eta)|}  \rd{\bx} = -\frac{d}{2(d+1)} + \frac{W_{1M}}{2}\,, \\
	\end{split}\end{equation}
	where in the inequality we used the fact that for any $\bx\in B(\bx_{*,j};\eta)$, we have
	\begin{equation}\begin{split}
		\Big(W_1*\frac{ \chi_{B(\bx_{*,i};\eta)}}{|B(0;\eta)|}\Big)(\bx) = & \frac{1}{|B(0;\eta)|}\int_{B(\bx_{*,i};\eta)} W_1(\bx-\by)\rd{\by} \\
		= & \frac{1}{|B(0;\eta)|}\int_{B(\bx-\bx_{*,i};\eta)} W_1(\by)\rd{\by} \le W_{1M}\,.
	\end{split}\end{equation}
	since $|\bx-\bx_{*,i}|<1+\eta$. Therefore, under the assumption \eqref{thm_nonradgen_1}, we see that $E[\rho_{*,\eta}] < \inf_{\rho\in \cM(\mathbb{R}^d),\,\textnormal{radial}} E[\rho]$, and thus no minimizer of $E$ is  radial by Lemma \ref{lem_radmin}.
	
\end{proof}

As a special example for Theorem \ref{thm_nonradgen}, we give the following construction. Denote
\begin{equation}
	\psi(x) = \exp(-\frac{1}{1-x^2})\chi_{|x|<1}\,,
\end{equation}
as a smooth nonnegative even function whose support is $[-1,1]$ and decreasing on $[0,1]$, and
\begin{equation}
	\phi(x) = \frac{\int_{-\infty}^x \psi(y)\rd{y}}{\int_{-\infty}^\infty \psi(y)\rd{y}}\,,
\end{equation}
which is a smooth increasing function with $\phi=0$ on $(-\infty,-1]$ and $\phi=1$ on $[1,\infty)$. We also have $\phi''>0$ on $(-1,0)$ and $\phi''<0$ on $(0,1)$.
\begin{theorem}\label{thm_nonradex}
	Assume $d\ge 2$. Fix $\epsilon<\epsilon_0$, $d-2<s<d$. For sufficiently small $\alpha,\beta>0$ (with $\beta<\epsilon$), construct (with notation $r=|\bx|$)
	\begin{equation}\label{thm_nonradex_1}
		W(\bx) = \alpha r^{-s}(1-\phi(4r-7)) + \left\{\begin{split}
			0, & \quad r\le 1-\epsilon \\
			-\phi\Big(\frac{r-(1-\epsilon+\frac{\beta}{2})}{\beta/2}\Big),& \quad 1-\epsilon < r < 1-\epsilon+\beta \\
			-1, & \quad 1-\epsilon+\beta \le r \le 1+\epsilon-\beta \\
			-1+\frac{\phi(r-(2+\epsilon-\beta))}{\phi(-1+\beta)}, & \quad 1+\epsilon-\beta<r<3+\epsilon-\beta \\
			-1+\frac{1}{\phi(-1+\beta)}, & \quad r\ge 3+\epsilon-\beta
		\end{split}\right.\,,
	\end{equation}
	which is a radial function, satisfying the assumptions of Theorem \ref{thm_exist2}, smooth on $\mathbb{R}^d\backslash\{0\}$, repulsive at short distance and attractive at long distance. Every minimizer of $E$ with this potential $W$ is a H\"older continuous function which is not radially symmetric.
\end{theorem}

The construction in \eqref{thm_nonradex_1} is illustrated in Figure \ref{fig1}.

\begin{figure}[htp!]
	\begin{center}
		\includegraphics[width=0.99\textwidth]{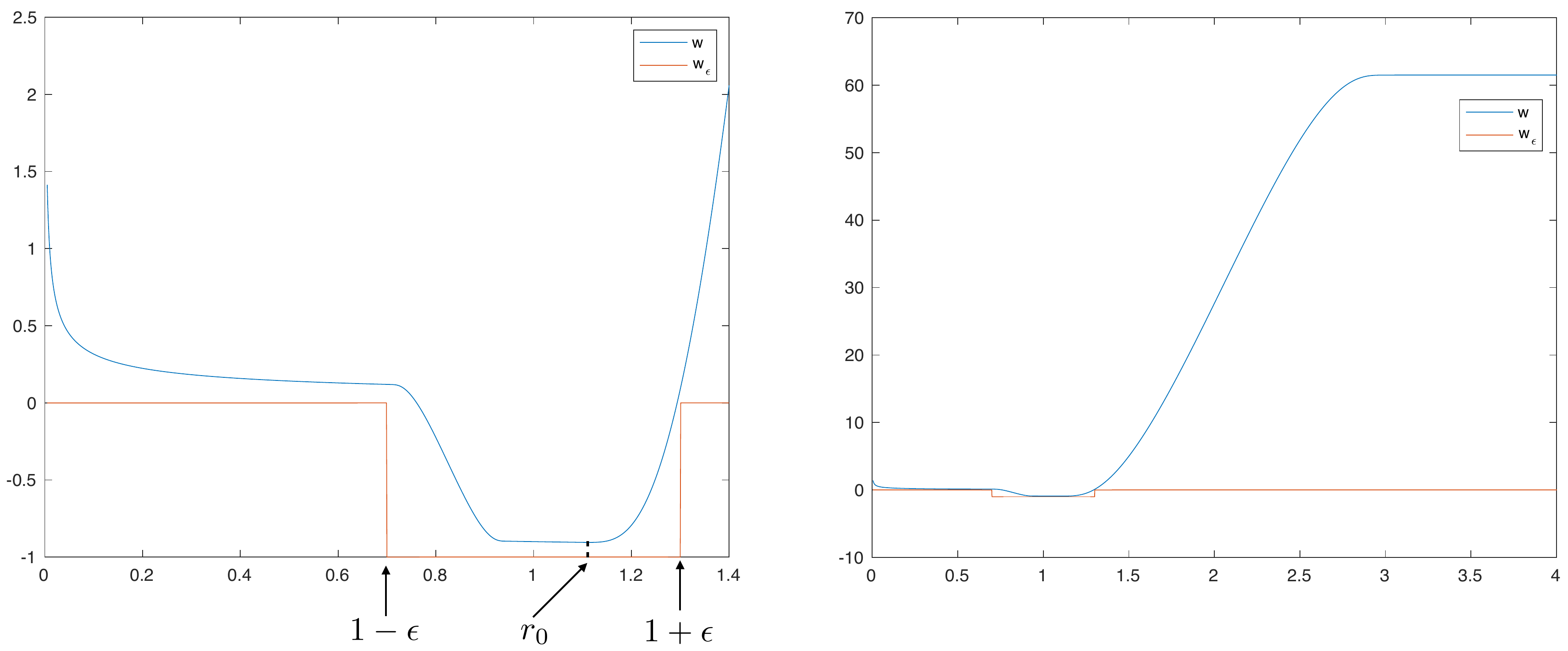}
		\caption{Illustration of the construction in \eqref{thm_nonradex_1}, as a function of $r$. For the clarity of illustration, the parameters $\epsilon,\alpha,\beta$ as in the picture may not be small enough to satisfy the assumptions of Theorem \ref{thm_nonradex}. }
		\label{fig1}
	\end{center}	
\end{figure}

\begin{proof}
	We first analyze the property of $W(\bx) = w(r) = \alpha r^{-s}(1-\phi(4r-7)) + w_2(r)$, where $w_2$ is given by the piecewisely-defined function as in \eqref{thm_nonradex_1}. One can easily check that at $r=1-\epsilon, 1-\epsilon+\beta, 1+\epsilon-\beta, 3+\epsilon-\beta$, the connections between pieces are smooth, and thus $W$ is smooth on $\mathbb{R}^d\backslash\{0\}$. On $1-\epsilon < r < 1-\epsilon+\beta$, $w_2$ is decreasing. On $1+\epsilon-\beta<r<3+\epsilon-\beta$, $w_2$ is increasing. Also notice that $w_2(1+\epsilon)=0$. It is clear that $W$ satisfies {\bf (W0)} with $\lim_{|\bx|\rightarrow\infty}W(\bx) = -1+\frac{1}{\phi(-1+\beta)}$ which is large if $\beta$ is small, and thus satisfies the assumptions of Theorem \ref{thm_exist2} for sufficiently small $\alpha,\beta$.
	
	Then we prove that $W$ is repulsive at short distance and attractive at long distance, i.e., there exists $r_0>0$ such that $w'(r)<0$ for $0<r<r_0$ and $w'(r)\ge 0$ for $r>r_0$. This is done by separating into several ranges of $r$.
	\begin{itemize}
		\item First notice that since $s>0$, we have $(r^{-s})' < 0$ and $(r^{-s})'' > 0$ for any $r>0$. Therefore $w'(r)<0$ for $0<r\le 1+\epsilon-\beta$ since $1+\epsilon-\beta<\frac{3}{2}$ and $1-\phi(4r-7)=1$ for such $r$. 
		\item Then notice that $w_2$ is convex on $(1+\epsilon-\beta,\frac{3}{2})$ since $\frac{3}{2}<2+\epsilon-\beta$. Therefore $w''(r)>0$ for $1+\epsilon-\beta<r<\frac{3}{2}$ since $1-\phi(4r-7)=1$ for such $r$. Since $w'(1+\epsilon-\beta)<0$ and 
		\begin{equation}\label{dw32}
			w'\Big(\frac{3}{2}\Big) = -\alpha s\Big(\frac{3}{2}\Big)^{-s-1} + \frac{\phi'(\frac{3}{2}-(2+\epsilon-\beta))}{\phi(-1+\beta)}  > 0\,,
		\end{equation}
		for sufficiently small $\alpha,\beta$, there exists a unique $r_0\in (1+\epsilon-\beta,\frac{3}{2})$ such that $w'(r_0)=0$, and $w'(r)<0$ for $1+\epsilon-\beta<r<r_0$, $w'(r)>0$ for $r_0<r\le \frac{3}{2}$. 
		\item A similar calculation as \eqref{dw32} (involving $\phi(4r-7)$ and its derivative) shows that $w'(r)>0$ for $\frac{3}{2}<r\le 2$ for sufficiently small $\alpha,\beta$.
		\item For $r>2$ it is clear that $w'(r)\ge 0$ because $w_2'(r)\ge 0$ and $1-\phi(4r-7)$ vanishes.
	\end{itemize}
	Therefore  $W$ is repulsive at short distance and attractive at long distance.
	
	Denote $W_2(\bx)=w_2(|\bx|)$. Then we write 
	\begin{equation}
		W = W_\epsilon + \alpha r^{-s}(1-\phi(4r-7)) + W_3\,,
	\end{equation}
	where $W_3 = W_2-W_\epsilon$ satisfies $\supp W_3 = \{\bx:1-\epsilon\le |\bx| \le 1-\epsilon+\beta\text{ or }|\bx|\ge 1+\epsilon-\beta\}$, with
	\begin{equation}
		0\le W_3(\bx) \le 1,\quad \forall 1-\epsilon\le |\bx| \le 1-\epsilon+\beta\text{ or }1+\epsilon-\beta \le |\bx| \le 1+\epsilon\,,
	\end{equation}
	where we used $w_2(1+\epsilon)=0$. Therefore $W_1 := W-W_\epsilon = \alpha r^{-s}(1-\phi(4r-7)) + W_3 \ge 0$. Then we estimate
	\begin{equation}\begin{split}
		 \frac{1}{|B(0;\eta)|}\int_{B(\bx;\eta)} & \alpha |\by|^{-s}(1-\phi(4|\by|-7))\rd{\by} \\
		\le & \frac{1}{|B(0;\eta)|}\int_{B(\bx;\eta)} \alpha |\by|^{-s}\rd{\by} \le \alpha\frac{1}{|B(0;\eta)|}\int_{B(0;\eta)}|\by|^{-s}\rd{\by} \\
		= & \alpha \frac{|S^{d-1}|}{|B(0;1)|}\eta^{-d} \int_0^\eta r^{-s+d-1}\rd{r} = \alpha \frac{|S^{d-1}|}{(d-s)|B(0;1)|}\eta^{-s}\,,
	\end{split}\end{equation}
	for any $\bx\in\mathbb{R}^d$. Also, provided $\eta\le \frac{\epsilon}{2}$ and $|\bx|<1+\eta$, we have
	\begin{equation}\begin{split}
		\frac{1}{|B(0;\eta)|}\int_{B(\bx;\eta)} & W_3(\by)\rd{\by} \le  \frac{1}{|B(0;\eta)|}\int_{B(0;1+\epsilon)} W_3(\by)\rd{\by} \\
		\le & \frac{|S^{d-1}|}{|B(0;\eta)|}\Big(\int_{1-\epsilon}^{1-\epsilon+\beta}r^{d-1}\rd{r} + \int_{1+\epsilon-\beta}^{1+\epsilon}r^{d-1}\rd{r}\Big) \\
		\le & \beta\frac{|S^{d-1}|}{|B(0;\eta)|}\cdot 2(1+\epsilon)^{d-1}\,.
	\end{split}\end{equation}
	Therefore, fixing the choice $\eta=\frac{\epsilon}{2}$, then for sufficiently small $\alpha,\beta$, \eqref{thm_nonradgen_1} is true, which implies that no minimizer of $E$ is radially symmetric by Theorem \ref{thm_nonradgen}. The fact that any minimizer is a H\"older continuous function is a consequence of \cite[Theorem 3.10]{CDM} since $d-2<s<d$ and $W(\bx)-\alpha|\bx|^{-s}$ is smooth on $\mathbb{R}^d$ and $W$ is constant for $|\bx|\ge 3+\epsilon-\beta$ (the condition (H2s) in \cite{CDM} can be made true by further modifying $W$ outside a large ball, which does not affect the minimizers due to the compact support property of any minimizer from \cite[Theorem 1.4]{CCP15}).
	
\end{proof}

\section*{Acknowledgement}

The author would like to thank Jiuya Wang for inspiring discussions. The author would like to thank Ryan W. Matzke for communications on the existing results on the break of symmetry for energy minimizers.

\bibliographystyle{alpha}
\bibliography{minimizer_book_bib.bib}

\end{document}